\documentclass[12pt, reqno]{amsart}
\usepackage{latexsym}

\usepackage[latin1]{inputenc}
\usepackage[colorlinks]{hyperref}

\usepackage{amssymb,amsfonts,amsmath,amsthm,graphicx}
\usepackage{epstopdf}
\usepackage{comment}

\newtheorem{thm}{Theorem}
\newtheorem{lem}{Lemma}
\newtheorem{rem}{Remark}

\newcommand{\setD}{{\mathord{\mathbb D}}}

\newcommand{\setC}{{\mathord{\mathbb C}}}

\newcommand{\setR}{{\mathord{\mathbb R}}}
\newcommand{\RR}{\mathbb{R}}

\newcommand{\ol}{\overline}

\begin{document}
\title{A four--dimensional Neumann ovaloid}
\author[L. Karp]{Lavi Karp}
\author[E. Lundberg]{Erik Lundberg}

\address{%
Department of Mathematics\\ ORT Braude College\\
P.O. Box 78, 21982 Karmiel\\ Israel}

\email{karp@braude.ac.il}

\address{%
Department of Mathematical Sciences\\ Florida Atlantic University, Boca Raton, 
FL 33431
\\ USA}

\email{elundber@fau.edu}



\begin{abstract}
What is the shape of a uniformly massive object 
that generates a gravitational potential
equivalent to that of two equal point-masses?
If the weight of each point-mass is sufficiently small compared to the distance between
the points then the answer is a pair of balls
of equal radius, one centered at each of the two points,
but otherwise it is a certain domain of revolution
about the axis passing through the two points.
The existence and uniqueness of such a domain is known,
but an explicit parameterization is known only in the plane
where the region is referred to as a Neumann oval.
We construct a four-dimensional ``Neumann ovaloid'',
solving explicitly this inverse potential problem.

\end{abstract}

\date{} 
\maketitle
%\subjclass[2010]{Primary  31B05, 30E20; Secondary 35R35}

% 
% 

\section{Introduction}
A  domain $\Omega \subset \setR^n$
is called a \emph{quadrature domain} 
if it admits a formula for the integration of any harmonic and integrable  
function $u $ in $\Omega$,
\begin{equation}\label{eq:QFgen}
 \int_{\Omega} {u \, dV} = \langle T, u \rangle,
\end{equation}
where $T$ is a distribution (independent of $u$) such that 
$T \big| _{\setR^n\setminus\Omega}=0$.
In particular, when $T$ is a measure, then by applying (\ref{eq:QFgen}) to the 
Newtonian kernel it results that the external potential of the body $\Omega$ 
with density one is equal to the potential of the measure $T$.
If $T$ is a finitely-supported distribution of finite-order
(so the right-hand-side of (\ref{eq:QFgen}) is a finite sum of weighted point 
evaluations
of $u$ and its partial derivatives), then
$\Omega$ is referred to as a \emph{quadrature domain in the classical sense}.
There are many examples of quadrature domains in the classical sense  in the 
plane, where conformal mappings can be used to 
construct them (see e.g. \cite{Davis_74}). But  
in higher dimensional spaces, very few explicit examples are known 
\cite{EL2012, Karp92}, 
and the only known explicit example involving simple point evaluations
(with no partial derivatives appearing) is the ball. 
There are however existence results for such quadrature domains 
\cite{Gustafsson_90, Sakai_82}. 

\subsection{Neumann's oval}

One of the simplest non-trivial examples of a quadrature domain (in the 
classical sense) 
in the plane is the region whose boundary is described by the real algebraic 
curve (excluding the origin):
\begin{equation}
\label{eq:Neumannoval}
 (x^2+y^2)^2=\alpha^2(x^2+y^2)+4 \epsilon^2 x^2.
\end{equation} 

This curve is referred to as Neumann's oval 
(it is also known as the hypopede of Booth).
Denoting this region by $\Omega$, the quadrature formula  is a sum of two point 
evaluations:
\begin{equation*}
 \int_{\Omega} {u \, dA} =\pi B \cdot u (-\epsilon,0) + \pi B \cdot u 
(\epsilon,0),
\end{equation*}
where the coefficient $B$ is a function of $\alpha$ and $\epsilon$. This 
quadrature identity was discoved by C. Neumann in 1908 \cite{Neumann},
see \cite[Ch. 5, 14]{Davis_74} or \cite[Ch. 3]{Shapiro} for details.

\subsection{A four-dimensional Neumann ovaloid}

More generally, we refer to a domain $\Omega \subset \setR^n$ as a 
\emph{Neumann 
ovaloid}
if it admits a quadrature formula having two quadrature nodes with equal 
weights.
We will consider the case of a four-dimensional Neumann ovaloid $\Omega \subset 
\setR^4$
satisfying:
\begin{equation}
\label{eq:QF}
 \int_{\Omega} {u dV} = \pi^2 A\left\{ u(-\epsilon,0,0,0) +  u 
 (\epsilon,0,0,0)\right\},
\end{equation}
for some positive constant $A$.

As one should expect, $\Omega$ is axially-symmetric (see the next paragraph below).
However, the axially-symmetric domain in $\setR^4$ generated by the
 rotation of a two-dimensional Neumann oval is not a Neumann ovaloid,
and in fact, it is not even a quadrature domain in the classical sense.
Instead, 
as the first author showed in \cite{Karp92},
it has a quadrature formula supported on the whole segment joining the original 
quadrature points.

Concerning uniqueness, suppose 
that the distribution $T$ of (\ref{eq:QFgen}) is a non--negative 
measure with compact support in a hyperplane, as in the case of interest 
(\ref{eq:QF}). If  $\Omega$ is a bounded quadrature 
domain  for $T$, then $\Omega$ is symmetric with respect to the same hyperplane 
and the complement of the closure of $\Omega$ is connected. This result was 
proved by Sakai in the plane  \cite[\S 14]{Sakai_82}, and for the $n$-dimensional 
case see \cite[\S 4]{Shapiro}.
We may apply this to every hyperplane passing through the two quadrature points
in order to conclude that 
the Neumann ovaloid $\Omega$ is axially symmetric

Since the Neumann ovaloid $\Omega$ is axially symmetric
and has connected complement of its closure,
it must be generated by rotation of a simply connected planar domain $D_p$.
Invoking the Riemann mapping theorem, 
there is a conformal map $f$ from the unit disk $\setD$ to $D_p$,
and $f$ is unique once the value of $f(0)$ and $\arg f'(0)$ are prescribed.
In order to construct $\Omega$, it thus suffices to 
determine explicitly the conformal map $f$.

%The following is our main result.

\begin{thm}
\label{thm:NeumannOvaloid}
Let $\Omega \subset \setR^4$ be the quadrature domain that satisfies the 
formula 
(\ref{eq:QF}),
and let $D_p$ denote the simply connected  domain that generates $\Omega$ 
by rotation.
Let $f$ be the conformal map from the unit disk $\setD$ to $D_p$ such that
$f(0)=0$ and $f'(0)>0$.
Then $f$
is given by:
\begin{equation}
\label{eq:conf}
f(\zeta) = \frac{C}{2 \pi i}\int\limits_{\{|w| = 1\}} \frac{1}{w - \zeta} 
\dfrac{(w^2-1) \sqrt{(w^2 + a^2) (1 + a^2 w^2)}}{(w^2-b^2)(1-b^2w^2)} dw, 
\end{equation}
for some real-valued positive constants $a,b$ and  $C$. 

\end{thm}

\begin{rem}
 There are existence results for point mass quadrature domains \cite{Gust81, 
Gustafsson_90, Sakai_82}; the important attribute of Theorem 
\ref{thm:NeumannOvaloid} is the explicit formula for the domain.  For each 
$\epsilon$ and $A$, the Neumann ovaloid (\ref{eq:QF}) is unique, this follows 
from the 
uniqueness of bounded  quadrature domains for non--negative measures with 
compact support in a hyperplane (see \cite[]{Sakai_82, Shapiro}). 
\end{rem}

In Section \ref{sec:prelim},
we review some essential background 
on axially symmetric quadrature domains in $\setR^4$
following \cite{EL2012, Karp92}. 
We prove 
Theorem \ref{thm:NeumannOvaloid}  in Section \ref{sec:main}.
The proof is based on the 
approach introduced in \cite{EL2012}
where A. Eremenko and the second author
used it to give a negative answer to the question 
of H.S. Shapiro \cite{Shapiro}
on the algebraicity of quadrature domains in $n>2$ dimensions.
In Section \ref{sec:constants} we 
describe how $a$ and $b$  of formula (\ref{eq:conf}) are related when $C=1$ is 
fixed, and we show that the quadrature formula for the ball is recovered from 
(\ref{eq:QF}) in the limit as $a$ and $b$ tend to zero. In Section 
\ref{sec:numerical}, we present some graphics based on numerical implementation 
of Theorem \ref{thm:NeumannOvaloid}.

\section{Preliminaries}\label{sec:prelim}
%\label{prel}

Here we review an algebraic technique that applies to axially 
symmetric 
quadrature domains in $\RR^4$. For further details see \cite{EL2012, Karp92}.
An equivalent definition of  the  quadrature domain  (\ref{eq:QFgen})  
is by the free boundary problem
\begin{equation}
\label{eq:SP}
\left\{
\begin{array}{ll}
\Delta u = T, 
& \text{in}\ \Omega \\
u =\frac{1}{2}{|x|^2}, \nabla u  = x & \text{on} \ \partial\Omega\\
%\nabla u & = x, & \text{on} \ \partial \Omega
\end{array}\right..
\end{equation}

The two notions 
are equivalent when the domain $\Omega $ is bounded and the distribution $T$ is
compactly supported in $\Omega$. 
A solution $u$ to the overdetermined system (\ref{eq:SP}) is called 
the \textit{Schwarz potential} of $\Omega$.

In the case of   the Neumann ovaloid (\ref{eq:QF}), the 
distribution $T$, as well as the domain  $\Omega$ has axial symmetry about the 
$x_1$--axis.  Hence the Schwarz potential $u$ can be represented by a function 
of two variables, namely 
\begin{equation*}
 u(x_1,x_2,x_3,x_4)=U(X,Y),
\end{equation*}
where 
$Y=\sqrt{x_1^2+x_2^2+x_3^2}$ and   $U(X,Y)=U(X,-Y)$ for negative $Y$.  Then the 
function $U$ satisfies the free boundary problem 
\begin{equation}
 \label{eq:free-b}
 \left\{
\begin{array}{ll}
\Delta U+2Y^{-1}\frac{\partial U}{\partial Y}  = T,
 & \text{in}\ D_p \\
U =\frac{1}{2}({X^2+Y^2}), \nabla U = (X,Y), &  \text{on} \ \partial D_p
\end{array}\right.,
\end{equation}
where $D_p\subset \setR^2$ is the domain whose rotation generates  $\Omega$.

The main idea is to set  $V(X,Y)=YU(X,Y)$, then $V$ satisfies the Cauchy 
problem 
\begin{equation}
 \label{eq:Cauchy}
\left\{
\begin{array}{ll}
\Delta V  = 0,
 & \text{near}\ \partial D_p \\
V =\frac{Y}{2}({X^2+Y^2}), &  \text{on} 
\ \partial D_p\\ 
\nabla V = (XY,\frac{1}{2}(X^2+3Y^2)), &  
\text{on} 
\ \partial D_p
\end{array}\right..
\end{equation}
This enables us to solve the free boundary problem (\ref{eq:free-b}) by means 
of the Schwarz function. Indeed, letting $z=X+iY$ and 
$V_z=\frac{1}{2}\left(\frac{\partial V}{\partial 
X}-i\frac{\partial V}{\partial Y}\right)$, then on the boundary $\partial D_p$, 
\begin{equation*}
 V_z(X,Y)=\frac{1}{2}XY-i\frac{1}{4}(X^2+3Y^2)=\frac{i}{4}(\bar z^2-2\bar z z),
\end{equation*} 
where $\bar z=X-iY$. 
Next, we  replace $\bar z$ by 
$S(z)$  on the boundary $\partial D_p$,
and we obtain
\begin{equation}
\label{eq:int}
 V(z)=2\text{Re}\left(\int V_z(z,S(z))dz\right)
\end{equation}
as the solution to the Cauchy problem (\ref{eq:Cauchy}) (see \cite[Lemma 
2.3]{Karp92}). 
The representation of the solution $V$ by (\ref{eq:int}) implies that
the singularities of $V_z(z,S(z))=\frac{i}{4}\left(S^2(z)-2 z 
S(z)\right)$ determine the distribution $T$ in (\ref{eq:SP}). 
Since we are interested only in the structure of the distribution $T$, we may 
perturb this expression by a holomorphic function.
Thus we 
conclude that the singularities of the  expression
\begin{equation}
\label{eq:9}
 \frac{i}{4}\left(S^2(z)-2 zS(z)\right)+\frac{z^2}{4} 
=\frac{i}{4}\left(S^2(z)-z\right)^2
\end{equation} 
govern the distribution $T$.  

Due to the axial symmetry, the support of the distribution $T$ is on the 
$x_1=X$--axis. So if the expression in (\ref{eq:9}) has non--vanishing 
residue at 
the points $(\pm \epsilon,0)$, then   in the integration (\ref{eq:int}) causes 
to a logarithmic term and consequently the support of $T$  is on a segment 
joining 
these two points. This implies that $\Omega$ will not be point masses 
quadrature domain.  For 
example, if $D_p$ is the Neumann oval (\ref{eq:Neumannoval}), then 
\begin{equation*} 
\begin{split}
&\frac{i}{4}\left(S^2(z)-z\right)^2 =  i\frac{(\alpha^2+2\epsilon^2)^2}{16}
\left(\frac{1}{(z-\epsilon)^2} 
+\frac{1}{(z+\epsilon)^2}\right)\\ + & i\frac{\alpha^4}{ 
16\epsilon}\left(\frac{1}{z-\epsilon}-\frac{1}{z+\epsilon}\right)+ h(z),
\end{split}
\end{equation*} 
where $h$ is holomorphic in $D_p$. 

However, if we require that 
\begin{equation}
\label{eq:Schwarz} 
\frac{1}{2}\left(S(z)-z\right)^2=\dfrac{A}{(z-\epsilon)^2}+\dfrac{A}{
(z+\epsilon)^2}+h(z),
\end{equation} 
then 
\begin{equation*}
\begin{split}
 V(X,Y)&=2\text{Re}\left(\int V_z(z,S(z))dz\right)\\ 
& =\frac{-AY}{(X-\epsilon)^2+Y^2}+
\frac{-AY}{(X+\epsilon)^2+Y^2}+H(X,Y), 
\end{split}
 \end{equation*}
with $H$ harmonic in $D_p$. 
Recalling that $U(X,Y)=Y^{-1}V(X,Y)$, we see that the singular part of the 
Schwarz potential (\ref{eq:SP}) comprises the expression
$\frac{A}{(x_1\pm\epsilon)^2+x_2^2+x_3^2+x_4^2}$, which is the fundamental 
solution of the 
Laplacian in $\setR^4$ at the points $(\pm\epsilon,0,0,0)$. Hence
\begin{equation*}
 T=\pi^2A\left(\delta_{(-\epsilon,0,0,0)}+\delta_{(+\epsilon,0,0,0)}\right)
\end{equation*}
is the distribution of the Schwarz potential (\ref{eq:SP}) and consequently 
the rotation of $D_p$ yields the Neumann Ovaloid (\ref{eq:QF}), here $\delta$ 
denote the Dirac measure.

% Since the Laplacian and boundary conditions are rotationally-invariant, 
% and the measure on the right hand side of (\ref{eq:SP}) 
% is invariant under rotations in the last three variables, a
% solution to (\ref{eq:SP}) has axial symmetry (see \cite{EL2012, Karp92}), and  
% hence $\Omega$ is axially-symmetric. 

%\textcolor{red}{To be continued, and to derive (\ref{eq:Schwarz}).}

\section{Construction of the Neumann Ovaloid}\label{sec:main}

%We denote points in
%$\setR^4$ by $(x_1,x_2,x_3,x_4)$, let $x=x_1$ and $y=\sqrt{x_2^2+x_3^2+x_4^2}$.

As in the statement of Theorem \ref{thm:NeumannOvaloid}, take $f$ to be the 
conformal map from the 
unit disk $\setD$ into $D_p$ such that $f(0) = 0$ and $f'(0) > 0$ is real. With 
this normalization, the conformal map is unique. By symmetry of the domain $D_p$ 
under complex conjugation, and uniqueness of the conformal map, $f$ is real, 
i.e., $f^*(w) := \ol{f(\ol{w})}=f(w)$.  

% By the prescribed quadrature formula (\ref{eq:QF}) and the solution to the 
% system (\ref{eq:SP}),
% the Schwarz function  $S(z)$ of $\partial D_p$ must satisfy (see 
% \cite{EL2012, Karp92})
% \begin{equation}
% \label{eq:Schwarz} 
% \frac{1}{2}\left(S(z)-z\right)^2=\dfrac{A}{(z-\epsilon)^2}+\dfrac{A}{
% (z+\epsilon)^2}+h(z),
% \end{equation} 
% where $h$ is holomorphic in $D_p$. 

Let $S$ be the Schwarz function of the boundary $\partial D_p$. Following \cite{EL2012}, 
we consider the pullback of (\ref{eq:Schwarz})
under the conformal map $f$.
Using the relation $S(f(w)) = f^*(1/w) = f(1/w),$
we have 
\begin{equation}
\label{eq:pullback}
\frac{1}{2}\left(f(w)-f\left(\frac{1}{w}\right)\right)^2 =
\dfrac{A}{(f(w)-\epsilon)^2}+\dfrac{A}{(f(w)+\epsilon)^2}+h(f(w)).
\end{equation} 

Setting
\begin{equation}
\label{eq2}
 g(w)=\left(f(w)-f\left(\frac{1}{w}\right)\right)^2,
\end{equation} 
then obviously $g(1/w) = g(w)$.
Using this fact,
we can use (\ref{eq:pullback}) to analytically continue $g(w)$
to the entire plane as a meromorphic (and in fact rational)
function. Namely, we prove the following:

\begin{lem}
\label{lem:1}
The function $g$ is a rational function of degree exactly $8$
and takes the form: 
\begin{equation}\label{eq:lemma}
g(w) = \frac{c(w^2-1)^2(w^2+a^2)(1+a^2w^2)}{(w^2-b^2)^2(1-b^2w^2)^2},
\end{equation}
with $a$, $b$, $c$ real constants and $a,b\neq\pm1$.
\end{lem}

\begin{proof}[Proof of Lemma \ref{lem:1}]
 
The points $\pm b\in \setD$ are the preimages of $\pm\epsilon\in D_p$ under $f$.
From (\ref{eq:pullback}) $g$ has a pole of order two at $\pm b$
and no other poles in $\setD$.
Since
$g\left( \frac{1}{w} \right)=g(w)$, $g$ has also poles of order two at 
$\pm 
\frac{1}{b}$. 
There are no other poles in $\setC \setminus \setD$, since otherwise 
we would have a contradiction to (\ref{eq:Schwarz}).
Thus, $g$ is a meromorphic function in the entire plane with exactly four poles 
of order two.  If $g(0)=0$, then the conformal map $f(w)\to \infty$ as 
$w\to\infty$, and that would imply that the Schwarz function has a pole at the 
origin, which  contradicts (\ref{eq:Schwarz}).
This implies $g$ is a rational function of degree exactly $8$.

Having determined the location of the poles, we have
\begin{equation*}
\label{eq:4}
 g(w)=\dfrac{P(w)}{(w^2-b^2)^2(1-b^2w^2)^2},
\end{equation*} 
where $P$ is a polynomial of degree  $8$. 
From (\ref{eq2}) we see that $g(\pm1)=0$ and the 
zeros have order two. 
Hence $g$ has additional four zeros. So suppose $w_0$ is a zero of $g$, then 
$w_0\neq 0$ since $g(0)\neq0$. Note that $f(w)=-f(-w)$, which follows from the 
symmetry of $D_p$ with respect to the 
real and imaginary axes, hence $-w_0$ is  a zero of $g$,
and by (\ref{eq2}) $g$ vanishes  $\pm 1/w_0$.
Thus
$$P(w)=c(w^2-1)^2(w^2-w_0^2)(w^2-1/w_0^2).$$

Now set  $w=x+i y$ and $w_0=a e^{i\theta}$ for some positive $a$, then 

\begin{equation*}
(x^2-w_0^2)(x^2-1/w_0^2)=x^4-x^2(a^2 e^{i2\theta}+a^{-2} e^{-i2\theta}) +1.
\end{equation*}
Since $f(w)=f^*(w)=\overline{f(\bar w)}$, $g$ is non--negative on the real 
axis. 
Hence the above expression is real, and therefore
\begin{equation*}
a^2 e^{i2\theta}+a^{-2} e^{-i2\theta}=\overline{a^2 
e^{i2\theta}+a^{-2} e^{-i2\theta}}=a^2 e^{-i2\theta}+a^{-2} 
e^{i2\theta},
\end{equation*}
which results in the identity
\begin{equation*}
a^2\sin 2\theta=a^{-2}\sin 2\theta.
\end{equation*}
Thus either $a=1$ or $\theta =0$ or $\theta=\frac{\pi}{2}$. Note that  $a=1$ is 
impossible since  then $\sqrt{g}$ would have branch points on the unit circle 
and consequently  $f$ is not single valued.  If $\theta=0$, then 
\begin{align*}
(x^2-w_0^2)(x^2-1/w_0)^2 &= x^4-x^2(a^2 +a^{-2}) +1 \\
&=\left(x^2-\frac{a^2+a^{-2}}{2}\right)^2+1- \left(\frac{a^2+a^{-2}}{2}\right)^2
\end{align*}
is negative for some $x$ when $a\neq\pm 1$. Hence 
$\theta=\frac{\pi}{2}$ and then 
\begin{equation*}
(x^2-w_0^2)(x^2-1/w_0)^2=x^4+x^2(a^2 +a^{-2}) 
+1\geq x^4+2x^2+1=(x^2+1)^2>0.
\end{equation*}
Thus $w_0=ia$, and we have
\begin{equation*}
P(w)=c(w^2-1)^2(w^2+a^2)(1+a^2w^2)
\end{equation*}
and we have obtained (\ref{eq:lemma}).
\end{proof}

Applying the Lemma \ref{lem:1}, 
and taking the square root in equation (\ref{eq2}) we have:
\begin{equation*}
f(w)-f\left(\frac{1}{w}\right) = \sqrt{g(w)} :=
\frac{C(w^2-1)\sqrt{(w^2+a^2)(1+a^2w^2)}}{(w^2-b^2)(1-b^2w^2)},
\end{equation*} 
where $C=\sqrt{c}$.
Multiplying by $({2 \pi i (w-\zeta)})^{-1}$ and integrating with respect to 
$dw$ 
along the contour $\{|w|=1\}$, we obtain
$$\frac{1}{2 \pi i} \int\limits_{\{|w|=1\}} \frac{f(w) - f(1/w)}{ w-\zeta} dw =
\frac{1}{2 \pi i}\int\limits_{\{|w|=1\}} \frac{\sqrt{g(w)}}{ w-\zeta} dw.
$$
The term $\frac{f(1/w)}{ w-\zeta}$
integrates to zero since it is analytic in $\setC \setminus \setD$
and $$ \frac{f(1/w)}{ w-\zeta} = O(1/|w|^2), \quad \text{as } w \rightarrow 
\infty.$$
Since $f(w)$ is analytic in $\setD$, 
we may apply the Cauchy integral formula:
$$f(\zeta) = \frac{1}{2 \pi i} \int\limits_{\{|w|=1\}} \frac{f(w)}{ w-\zeta} dw 
=\frac{1}{2 \pi i}\int\limits_{\{|w|=1\}} \frac{\sqrt{g(w)}}{ w-\zeta} dw.
$$
Thus we have:
\begin{equation}
f(\zeta) = \frac{C}{2 \pi i}\int\limits_{\{|w| = 1\}} \frac{1}{w - 
\zeta} 
\dfrac{(w^2-1)
 \sqrt{(w^2 + a^2) (1 + a^2 w^2)}}{(w^2-b^2)(1-b^2w^2)} dw,
\end{equation}
and this completes the proof of Theorem \ref{thm:NeumannOvaloid}.

%\begin{rem}
%The quadrature domain depends on only two parameters, $\epsilon$ and $A$,
%while the conformal map $f$ 
%appears to depend on three real parameters, $a \geq 0$, $b \geq 
%0$, and $C > 0$.
%However, the value of $a$ is determined by the value of $b$. We chose $C=1$, 
%which gives that $f$ is the identity map when $a=b=0$.
%\end{rem}

\section{The relation between $a$ and $b$}\label{sec:constants}

Notice that the quadrature domain depends on only two parameters, $\epsilon$ 
and 
$A$,
while the conformal map $f$ 
appears to depend on three real parameters, $a \geq 0$, $b \geq 
0$, and $C > 0$.
However, fixing $C>0$, the value of $a$ is determined by the value of $b$.
In this section, we fix $C=1$,
which gives that $f$ is the identity map when $a=b=0$.
We will show how parameters $a$ and $b$ are related
and also show that as $a$ and $b$ tend to zero,
we recover the quadrature formula for the unit ball.

Notice that equation (\ref{eq:Schwarz}) 
provides a constraint relating $a$ and $b$.
Namely, we observe that the residue of $(S(z)-z)^2$ at $z=\pm\epsilon$ vanishes,
which is equivalent to ${{\rm Res}(gf':\pm b)=0}$.

Since the pole is of order two, we have:
\begin{equation}
\label{eq:res:1}
 {\rm 
 Res}(gf':b)=\dfrac{d}{d\zeta}\left((\zeta-b)^2g(\zeta)f'(\zeta)\right)\big|_{
\zeta=b}=0.
\end{equation}
Setting 
\begin{equation*}
H(\zeta,a,b)=(\zeta-b)^2g(\zeta)=
\frac{(w^2-1)^2(w^2+a^2)(1+w^2a^2)}{(w+b)^2(1-w^2b^2)^2},
\end{equation*}
then by   the formula of $ f$ (\ref{eq:conf}), 
\begin{equation}
\label{eq:res:2}
\begin{split}
 {\rm 
 	Res}(gf':b) &=\frac{1}{2\pi 
 	i}\int\limits_{\{|w|=1\}}\left[\dfrac{\frac{\partial 
H}{\partial 
\zeta}(b,a,b)}{(w-b)^2}+2\dfrac{ H(b,a,b)}{(w-b)^3}\right]\sqrt{g(w)} dw \\ &
=:F(a,b).
\end{split}
\end{equation}
Thus we conclude that the rotation of the domain $D_p$ generates the Neumann 
ovaloid 
(\ref{eq:QF}) if and only if $F(a,b)=0$ for some positive $a$ and $b$. 
We shall use the implicit function theorem in order to justify that 
identity (\ref{eq:res:2}) determines 
the parameter $a$ as a function of $b$.

To this end, the expressions $\frac{\partial H}{\partial 
\zeta}(b,a,b)$ and $ H(b,a,b)$ can be computed using symbolic computation 
software: 
%\begin{equation}
%\label{eq:com:1}
%\begin{split}
%&\frac{\partial H}{\partial \zeta}(\zeta,a,b)\Big|_{\zeta=b}=\\
%&\frac{\left( {b}^{6}+4{b}^{4}-{b}^{2}\right){a}^{4}+\left( 
%3{b}^{8}+4{b}^{6}-4{b}^{4}+4{b}^{2}+1\right) 
%{a}^{2}+({b}^{6}+4{b}^{4}-{b}^{2})}{4{b}^{11}+8{b}^{9}-8{b}^{5}-4{b}^{3}},
%\end{split}
%\end{equation}

\begin{equation}
\label{eq:com:1}
\begin{split}
&\frac{\partial H}{\partial \zeta}(\zeta,a,b)\Big|_{\zeta=b}=\\
&\frac{a^2+b^2(-4a^4+4a^2-1)+4b^4(a^4-a^2+1)+b^6(a^4+4a^2+4)+3b^8a^2}
{4{b}^{11}+8{b}^{9}-8{b}^{5}-4{b}^{3}},
\end{split}
\end{equation}

%\begin{equation}
%\label{eq:com:2}
%H(b,a,b)=
%\frac{{b}^{2}{a}^{4}+\left( {b}^{4}+1\right) 
%{a}^{2}+{b}^{2}}{4{b}^{6}+8{b}^{4}+4{b}^{2}}.
%\end{equation}

\begin{equation}
\label{eq:com:2}
H(b,a,b)=
\frac{a^2+{b}^{2}({a}^{4}+1)+b^4a^2}{4{b}^{6}+8{b}^{4}+4{b}^{2}}.
\end{equation}

It follows from the above computations the functional $F(a,b)$ is 
not continuous when $b=0$. Hence, we set  $a=\sqrt{\delta}$ and
\begin{equation*}
G(\delta,b)=b^3F(\sqrt{\delta},b).
\end{equation*}
Then obviously $F(a,b)=0$ for $(a,b)\in\setR^2_+$  if and only if 
$G(\delta,b)=0$.  We shall now compute the Taylor expansion of $G$ near the 
origin. From (\ref{eq:com:1}) and (\ref{eq:com:2}) we see  that
\begin{equation*}
G(\delta,0)=\frac{1}{2\pi 
i}\int\limits_{\{|w|=1\}}\left(-\frac{\delta}{4}\right)\frac{(w^2-1)
\sqrt{(w^2+\delta)(1+\delta w^2)}}{w^4}dw.
\end{equation*}
 Hence $G(0,0)=0$ and
 \begin{equation*}
 \begin{split}
 & \frac{\partial G}{\partial\delta}(0,0)=\lim_{\delta\to 
 0}\frac{1}{\delta}G(\delta,0) \\= &
 \lim_{\delta\to 
 0}\frac{1}{2\pi i}\int\limits_{\{|w|=1\}}\left(-\frac{1}{4}\right)\frac{(w^2-1)
 \sqrt{(w^2+\delta)(1+\delta w^2)}}{w^4}dw\\  
 = 
 &\frac{1}{2\pi  
i}\int\limits_{\{|w|=1\}}\left(-\frac{1}{4}\right)\frac{(w^2-1)w}{w^4}dw=-\frac{
1}{4}.
 \end{split}
 \end{equation*}
Since both  $b^3\frac{\partial H}{\partial \zeta}(b,0,b)$ 
and  $b^3H(b,0,b)$ are of order $O(b^2)$ for $b$ near zero, 
$\frac{\partial 
G}{\partial b}(0,0)=0$ and
\begin{equation*}
\begin{split}
\frac{\partial G}{\partial b}(0,b) &=\frac{1}{2\pi 
i}\int\limits_{\{|w|=1\}}\left[
\left(\frac{\frac{\partial^2}{\partial 
b\partial\zeta}\left(b^3H(b,0,b)\right)}{(w-b)^2}+2\frac{\frac{\partial}{
\partial
 b}\left(b^3H(b,0,b)\right)}{(w-b)^3}\right)\right.\\ 
&\left. \times\frac{(w^2-1)w}{(w^2-b^2)(1-b^2w^2)}\right]dw +O(b^2). % \\= &
%\int_{\{|w|=1\}}\left(\frac{1}{4}\right)\frac{(w^2-1)w
%}{w^4}dw=\frac{1}{4}.
  \end{split}
\end{equation*}
From (\ref{eq:com:1}) and (\ref{eq:com:2}) we that $\frac{\partial^2}{\partial 
b\partial\zeta}\left(b^3H(b,0,b)\right)=\frac{b}{2}+O(b^2)$ and  \\
$\frac{\partial}{\partial 
b}\left(b^3H(b,0,b)\right)=O(b^2)$, hence
\begin{equation*}
\frac{\partial^2 G}{\partial b^2}(0,0)=\lim_{b\to 0}\frac{1}{b}\frac{\partial 
G}{\partial b}(0,b)=\frac{1}{2\pi 
i}\int_{\{|w|=1\}}\left(\frac{1}{2}\right)\frac{(w^2-1)w
}{w^4}dw=\frac{1}{2}.
\end{equation*}
In a similar manner we have computed $\frac{\partial^2 G}{\partial b\partial 
\delta}(0,0)=\frac{\partial^2 G}{\partial \delta^2}(0,0)=0$ (see Appendix  
\ref{appendix}), hence
\begin{equation}
\label{eq:com:3}
G(\delta,b)=\frac{1}{4}\left(-\delta+b^2\right)+O(\delta^3)+O(b^3).
\end{equation}
Thus by the implicit function theorem there is a function $\delta(b)$ such that 
$G(\delta(b),b)=0 $ for $b$ near zero and positive. Moreover, we see from 
(\ref{eq:com:3})  that $\delta=\delta(b)=b^2+O(b^3)$. Hence the 
functional  $F(a,b)$, which is defined by (\ref{eq:res:2}), vanishes  
near the origin on a curve, in which $a\approx b $, or more precisely 
\begin{equation}
\label{eq:asym}
a^2=b^2+O(b^3).
\end{equation}

The coefficient $A$ of the quadrature identity (\ref{eq:QF}) can be computed as 
follows. When 
$F(a,b)=0$, then (\ref{eq:Schwarz}) holds and therefore by (\ref{eq:pullback})
\begin{equation*}
\begin{split}
A &=\frac{1}{4\pi 
i}\int\limits_{\{|z-\epsilon|=\rho\}}\left(S(z)-z\right)^2(z-\epsilon) dz\\ &=
\frac{1}{4\pi 
	i}\int\limits_{\{|w-b|=\tilde{\rho}\}}g(w)\left(f(w)-f(b)\right)f'(w) 
dw.
\end{split}
\end{equation*}
Writing $f(w)-f(b)=(w-b)\left(f'(b)+\varPhi(b)\right)$, where $\varPhi(b)=0$, 
then by Lemma \ref{lem:1},
\begin{equation*}
\begin{split}
 & A =A(b)= \\ & \frac{1}{4\pi 
	i}\int\limits_{\{|w-b|=\tilde{\rho}\}}\frac{(w^2-1)^2{(w^2+a^2)
(1+w^2a^2)}\left(f'(b)+\varPhi(b)\right)f'(w)}{(w+b)^2(1-w^2b^2)^2}
\frac{dw}{(w-b)}
\\ &
=\frac{1}{2}\frac{(b^2-1)^2{(b^2+a^2)(1+b^2a^2)}}{(2b)^2(1-b^4)^2}
\left(f'(b)\right)^2.
\end{split}
\end{equation*} 
The estimate  (\ref{eq:asym})  enables us to examine the asymptotic behavior as 
$ b$ 
goes to 
zero. We see that 
\begin{equation*}
\lim_{b\to 0}A(b)=\frac{1}{4}f'(0))^2,
\end{equation*}
and since when $b\to 0$, the conformal mapping $f$ tends to the identity, we 
conclude that  \begin{math}
\lim_{b\to 0}A(b)= A(0)=\frac{1}{4}
\end{math}. 
On the other hand, letting $\epsilon$ tend to zero in (\ref{eq:QF}), 
then the  right hand side goes to $2\pi^2 
A(0)u(0,0,0,0)=\frac{\pi^2}{2}u(0,0,0,0)$, and the domain $\Omega$ tends to the 
unit ball. Hence the limit coincides with the mean value property of harmonic 
functions.

\section{Numerics}\label{sec:numerical}

Numerical implementation of Theorem \ref{thm:NeumannOvaloid} requires 
determining 
appropriate choices of parameters.
The parameter $C$ simply scales the domain,
so that there is only a one-parameter family of different shapes.
While varying the choice of $b$,
we choose $a$ to satisfy the relation (\ref{eq:res:1}).
This is done numerically using Matlab.
We then use the condition $f(b) = \epsilon$ (where $C$ appears as a scalar)
in order to choose $C$ so that $\epsilon = 1$.
This leads to a one-parameter family of different shapes having the same 
``foci'' $\pm \epsilon = \pm 1$.
We note that this family can also be interpreted as a four-dimensional 
Hele-Shaw flow with two point sources (located at the foci).

We used Matlab to perform numerical integration of the Cauchy transform 
appearing in (\ref{eq:conf}).
This requires some care in checking that the branch of the square root is 
defined appropriately.
The radius of the contour of integration should also be increased slightly 
(without crossing any singularities
of the integrand);
this is a convenient way to avoid numerically integrating through the simple 
pole presented
by the fact that $\zeta$ is on the unit circle.

Having carried out these steps, we display some images of the profile curves 
for 
the 
resulting confocal Neumann ovaloids in Figure \ref{fig:Neumann}.

\begin{figure}[h!]
\center
    \includegraphics[scale=.38]{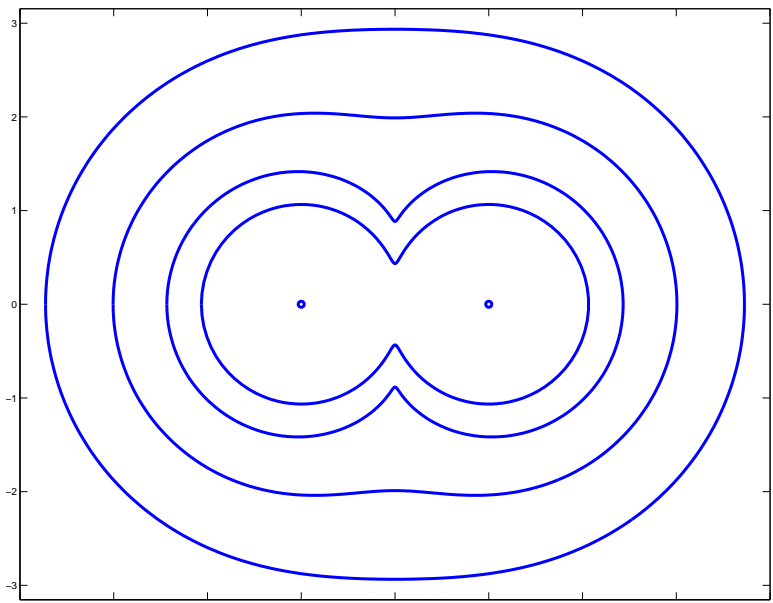}
    \caption{Profiles of some confocal ovaloids plotted using Matlab.
    Each curve is the image of the unit circle under
    a conformal mapping $f$ obtained as a numerical Cauchy transform.}
    \label{fig:Neumann}
\end{figure}

%\begin{comment}
\section{Appendix}
\label{appendix}

Here we shall provide further details for the
computations of the second order derivatives  $\frac{\partial^2 G}{\partial 
b\partial\delta}G(0,0)$ and
$\frac{\partial^2 G}{\partial^2\delta}G(0,0)$. 
From the expressions 
(\ref{eq:com:1}) and (\ref{eq:com:2}) we see that
\begin{equation}
\label{eq:app:2}
\begin{split}
\frac{\partial G}{\partial \delta}(\delta,b) &=
\frac{1}{2\pi i}\int\limits_{\{|w|=1\}}\left[\left(\frac{\delta 
+b^2(-4\delta^2+4\delta 
-1)}{-4(w-b)^2}+2\frac{b\delta}{4(w-b)^3}\right)\right. \\ & \times
\left.\frac{(w^2-1)(1+2\delta 
w^2+w^4)}{(w^2-b^2)(1-w^2b^2)2\sqrt{(w^2+\delta)(1+w^2\delta)}}\right]dw\\ & +
\frac{1}{2\pi i}\int\limits_{\{|w|=1\}}\left[\left(\frac{1 
+b^2(-8\delta+4)}{-4(w-b)^2}+2\frac{b}{4(w-b)^3}\right)\right. \\ & \times
\left.\frac{(w^2-1)\sqrt{(w^2+\delta)(1+w^2\delta)}}
{(w^2-b^2)(1-w^2b^2)}\right]dw 
+O(b^2)
\end{split}
\end{equation}
Letting $\delta=0$ and recalling that $\frac{\partial G}{\partial 
	\delta}(0,0)=-\frac{1}{4}$, we see that
\begin{equation*}
\begin{split}
&\frac{\partial G}{\partial \delta}(0,b)-\frac{\partial G}{\partial 
\delta}(0,0)\\ = &
\frac{1}{2\pi 
i}\int\limits_{\{|w|=1\}}\left[\left(\frac{-1}{4(w-b)^2}+\frac{2b}{4(w-b)^3}
\right)
\frac{(w^2-1)w}{(w^2-b^2)(1-w^2b^2)}\right]dw 
\\ & +\frac{1}{4}+O(b^2)\\ = &
\frac{1}{2\pi i}\int\limits_{\{|w|=1\}}\left[\left(\frac{2b}{4(w-b)^3}\right)
\frac{(w^2-1)w}{(w^2-b^2)(1-w^2b^2)}\right]dw \\ +&
\frac{1}{2\pi 
i}\int\limits_{\{|w|=1\}}\frac{1}{4}\left[(w^2-1)w\left(\frac{1}{w^4}-
\frac{1}{(w-b)^2(w^2-b^2)(1-w^2b^2)}\right)\right]dw \\ & +
O(b^2).
\end{split}
\end{equation*}
In the last equality  we used the fact that 
\begin{equation}
\label{eq:app:1}
\frac{1}{2\pi i}\int\limits_{\{|w|=1\}}\frac{(w^2-1)w}{w^4}dw=1.
\end{equation}
By Taylor expansion, 
\begin{equation*}
\frac{1}{w^4}-
\frac{1}{(w-b)^2(w^2-b^2)(1-w^2b^2)}=\frac{1}{w^4}\left(-\frac{2b}{w}
+O(b^2)\right),
\end{equation*}
and  hence
\begin{equation*}
\begin{split}
&\frac{\partial^2 G}{\partial b\partial\delta}(0,0)=\lim_{b\to 
0}\frac{1}{b}\left(\frac{\partial 
G}{\partial \delta}(0,b)-\frac{\partial G}{\partial 
	\delta}(0,0)\right)\\ &= 
\frac{1}{2\pi 
i}\int\limits_{\{|w|=1\}}\left[\left(\frac{2(w^2-1)w}{4w^5}\right)\right]dw
-\frac{1}{2\pi 
i}\int\limits_{\{|w|=1\}}\left[\left(\frac{2(w^2-1)w}{4w^5}\right)
\right]dw\\ & =0.
\end{split}
\end{equation*}
We turn now to the computation of $\frac{\partial^2 G}{\partial 
\delta^2}(0,0)$. 
Using 
(\ref{eq:app:2}) 
and 
(\ref{eq:app:1}), we have that 
\begin{equation*}
\begin{split}
&\frac{\partial G}{\partial \delta}(\delta,0)- \frac{\partial G}{\partial 
\delta}(0,0)\\ = &
\frac{1}{2\pi i}\int\limits_{\{|w|=1\}}-\frac{1}{8}
\left[\frac{\delta(w^2-1)(1+2\delta 
	w^2+w^4)}{w^4\sqrt{(w^2+\delta)(1+w^2\delta)}}\right]dw\\ + &
\frac{1}{2\pi i}\int\limits_{\{|w|=1\}}-\frac{1}{4}
\left[\frac{(w^2-1)\sqrt{(w^2+\delta)(1+w^2\delta)}}
{w^4}\right]dw +\frac{1}{4}\\ = &
\frac{1}{2\pi i}\int\limits_{\{|w|=1\}}-\frac{1}{8}
\left[\frac{\delta(w^2-1)(1+2\delta 
	w^2+w^4)}{w^4\sqrt{(w^2+\delta)(1+w^2\delta)}}\right]dw\\ + &
\frac{1}{2\pi i}\int\limits_{\{|w|=1\}}\frac{1}{4}
\left[\frac{(w^2-1)w}{w^4}\left(1-\sqrt{1+\delta\left(\frac{1}{w^2}+w^2\right)
	+\delta^2}\right)\right]dw.
\end{split}
\end{equation*}
Taylor expansion 
\begin{equation*}
\sqrt{1+\delta\left(\frac{1}{w^2}+w^2\right)
	+\delta^2}=1+\frac{\delta}{2}\left(\frac{1}{w^2}+w^2\right) + 
O(\delta^2)
\end{equation*}
yields that
\begin{equation*}
\begin{split}
 &\frac{\partial^2 G}{\partial 
\delta^2}G(0,0)=\lim_{\delta\to0}\frac{1}{\delta}\left(\frac{\partial 
G}{\partial 
\delta}G(\delta,0)-\frac{\partial G}{\partial \delta}G(0,0)\right)\\  = &
\frac{1}{2\pi i}\int\limits_{\{|w|=1\}}-\frac{1}{8}
\left[\frac{(w^2-1)(1+w^4)}{w^5}\right]dw\\ + &
\frac{1}{2\pi i}\int\limits_{\{|w|=1\}}\frac{1}{8}
\left[\frac{(w^2-1)w}{w^4}\left(\frac{1}{w^2}+w^2\right)\right]dw\\ = &
\frac{1}{8}-\frac{1}{8}=0.
\end{split}
\end{equation*}

%\end{comment}

% \bibliographystyle{amsplain}
%\bibliography{bibgraf-quad}

\end{document}